\documentclass[11pt]{amsart}

\usepackage{amssymb}

\usepackage{amsmath,latexsym,amsthm}

\usepackage{amsfonts}
\newtheorem{assumption}{Assumption}
\newtheorem*{theorem*}{Theorem}

\def\bel{\begin{equation}\label}
\def\eeq{\end{equation}}

\def\LT{L^1_{[0,T]}}

\newtheorem{remark}{Remark}[section]
\newtheorem{definition}{Definition}[section]

\newtheorem{theorem}{Theorem}[section]
\newtheorem{lemma}[theorem]{Lemma}
\newtheorem{proposition}[theorem]{Proposition}
\newtheorem{corollary}[theorem]{Corollary}

\def\ds{\displaystyle}

\def\bega{\begin{array}}
\def\enda{\end{array}}

\def\bepmatrix{\begin{pmatrix}}
\def\enpmatrix{\end{pmatrix}}

\def\bel{\begin{equation}\label}
\def\eeq{\end{equation}}
\newcommand\ee{\end{equation}}
\def\benl{\begin{equation*}}
\def\eenl{\end{equation*}}

\def\forall{\hbox{for all }~}
\def\be{\begin{equation}}
\def\beq{\begin{equation}}
\def\bel{\begin{equation}\label}
\def\eeq{\end{equation}}

\newcommand\ba{\begin{array}}
\newcommand\ea{\end{array}}

\def\vs{\vskip 2em}

\def\begi{\begin{itemize}}
\def\endi{\end{itemize}}

\def\pr{\partial}
\def\d{\text{d}}

\def\forall{\hbox{for all}~}

\newcommand{\cR}{\mathbb{R}}

\newcommand{\tras}{^\top}
\newcommand{\intT}{\int_0^T }

\newcommand{\dtt}{\mathrm{d}t}

\def\A{\mathcal{A}}
\def\B{\mathcal{B}}
\def\C{\mathcal{C}}

\def\G{\mathcal{G}}

\def\N{\mathcal{N}}
\def\P{\mathcal{P}}

\def\T{\mathcal{T}}

\def\W{\mathcal{W}}

\def\ah{\hat{a}}

\def\uh{\hat{u}}

\def\wh{\hat{w}}
\def\xh{\hat{x}}

\def\xih{\hat\xi}
\def\etah{\hat\eta}

\def\eps{\varepsilon}

\title[Commutative impulsive systems]{Necessary conditions involving Lie brackets for impulsive optimal control problems;\\ the commutative case$^1$}

\begin{document}

\vs
\author[M.S. Aronna]{Maria Soledad Aronna}
\address{M.S. Aronna, Dipartimento di Matematica ,
Universit\`a di Padova\\ Padova  35121, Italy}
\email{aronna@math.unipd.it}

\author[F. Rampazzo]{Franco Rampazzo}
\address{M.S. Aronna, Dipartimento di Matematica ,
Universit\`a di Padova\\ Padova  35121, Italy}
\email{rampazzo@math.unipd.it}

\maketitle

\footnotetext[1]{This work is supported by the European Union under the 7th Framework Programme «FP7-PEOPLE-2010-ITN»  Grant agreement number 264735-SADCO.}


\begin{abstract}
In this article we study control problems with systems that are governed by ordinary differential equations whose vector fields depend linearly in the time derivatives of some components of the control. The remaining components are considered as classical controls. This kind of system is called `impulsive system'.
We assume that the vector fields multiplying the derivatives of each component of the control are commutative. 

We use the results in Bressan and Rampazzo \cite{BreRam91}  where it is shown that the impulsive system can be reduced to a classical system of ordinary differential equations via a transformation of variables. The latter is used to give a concept of solution of the impulsive differential equation.

In \cite{BreRam91} they also provide maximum principles for both the original and the transformed optimal control problems. From these principles, we derive new necessary conditions in term of the adjoint state and the Lie brackets of the data functions.
\end{abstract}

\section{Introduction}\label{SecPb}

In this article we investigate necessary optimality conditions for a Mayer governed by the system
\begin{align}
\label{E}&\dot x(t) = \tilde{f}(x(t),u(t),a(t)) + \ds\sum_{i=1}^m  \tilde g_i(x(t),u(t))\dot u^i(t), \\
\label{E0}&(x,u)(0)=(x_0,u_0),
\end{align}
where $t\in [0,T],$ $x(t) \in \cR^n,$ $u(t)  \in U \subset \cR^m$ and $a(t) \in A \subset \cR^l.$ The detailed hypothesis concerning the control sets and the vector fields are given in Assumptions \ref{AssComm} and \ref{Ass} afterwards. 
Consider the vector fields  $f,g_\alpha:\cR^{n}\times U\rightarrow \cR^{n+m}$  for $\alpha =1,\dots,m,$ defined by
\bel{fieldsgf}
f\doteq f^j\frac{\partial}{\partial x^j},\qquad g_\alpha\doteq g_\alpha^j \frac{\partial}{\partial x^j} +     \frac{\partial}{\partial z^\alpha}\,,
\eeq
where  we have adopted the Einstein summation convention. Latin indexes run from $1$ to $n$ and Greek indexed run from $1$ to $m.$ The columns representations of these vector fields  are 
\bel{gi}
f = \bepmatrix \tilde{f^1}\, \\.\\ \tilde{f^n}\\ \,\\0\\\, \enpmatrix,\qquad
g_\alpha
=
\bepmatrix
{\tilde g_\alpha^1}\\.\\ {\tilde g_\alpha^n}
\\\,\\
{\bf e_\alpha}\\\,
\enpmatrix,
\ee
where ${\bf e_\alpha}$  is the $\alpha-$th. element of the  canonical basis of $\cR^m.$ The following hypothesis holds true along all the article and is of main importance.

\begin{assumption}[Commutativity]
\label{AssComm}
We assume that the vector fields $g_\alpha$ {\it commute}, i.e. for every pair $\alpha,\beta=1,\dots,m,$
\bel{gicomm}
{[{g}_\alpha,{g}_\beta]=0,}
\ee
where $[{g}_\alpha,{g}_\beta]$ is the {\em Lie bracket} of  $g_\alpha$ and $g_\beta$ defined by
\benl
 [{g}_\alpha,{g}_\beta] \doteq
 \left(\frac{\partial {g}_\beta ^i}{\partial x^j}
 {g}_\alpha^i -  \frac{\partial {g}_\alpha^i}{\partial x^j}{g}_\beta ^j + \frac{\partial {g}_\beta^i}{\partial u^\alpha} -  \frac{\partial {g}_\alpha^i}{\partial u^\beta} \right)  \frac{\partial}{\partial x^i}.
\eenl
\end{assumption}
Notice that $[{g}_\alpha,{g}_\beta]$ has zero $z$-components, since  all the 
${g}_\alpha$ have constant $z$-components.
 The column representation of $[{g}_\alpha,{g}_\beta]$ is
\benl
\begin{pmatrix}
\ds\nabla_x \tilde{g}_\alpha\, \tilde{g}_\beta - \nabla_x \tilde{g}_\beta\, \tilde{g}_\alpha + \frac{\pr \tilde{g}_\alpha}{\pr u_\beta}  - \frac{\pr \tilde{g}_\beta }{\pr u_\alpha} \\
0
\end{pmatrix}.
\eenl

The main result of the present paper is the one stated next and it is proved in Section \ref{SecNC}. Assume for the moment that $p$ is the adjoint state and that it can be defined as in the classical framework.
\begin{theorem*}[Necessary conditions involving Lie brackets]
Let $(x^*,u^*,a^*)$ be optimal for $\P,$  and let $i=1,\dots m$ be an index.  Then the following statements hold.
\begin{itemize}
\item[(i)] Let $t \in [0,T)$ be any time such that there exists $\sigma>0$ sufficiently small such that $u^*(t) +\sigma {\bf e_i} \in U$ a.e. on $[t,T].$ Then
\be
\label{NCi}
p(t) \cdotp g_i (x^*(t),u^*(t)) \leq 0.
\ee
\end{itemize}
Furthermore, for a.a. $t\in  [0,T],$ it holds:
\begin{itemize}
 \item[(ii)] If there exists $\sigma_0>0$ sufficiently small such that $u^*(t) +\sigma {\bf e_i} \in U$ for all $\sigma \in [0,\sigma_0],$ then
\be
\label{NCii}
p(t) \cdotp [g_i,f] (x^*(t),u^*(t),a^*(t)) \geq 0.
\ee
\item[(iii)] If there exist $h\in \cR^m$ and $\sigma_0>0$ sufficiently small such that $u^*(t) \pm \sigma {\bf e_i} \in U$ for all $\sigma \in [0,\sigma_0],$ then
\be
\label{NCiii}
p(t) \cdotp \sum_{j,k=1}^m h_jh_k [g_j,[g_k,f] \,] (x^*(t),u^*(t),a^*(t)) \geq 0.
\ee
In other words,
\be
h^\top Qh \geq 0,
\ee
where $Q$ is a symmetric matrix with entries
\benl
Q_{jk}(t) := p(t) \cdotp [g_j,[g_k,f] \,] (x^*(t),u^*(t),a^*(t)) .
\eenl
\end{itemize}
Here all the Lie brackets are computed in the variable $(x,u).$
\end{theorem*}

There is a wide literature concerning impulsive control systems, and many different approaches can be identified. In \cite{Ris65} Rishel derived necessary conditions for a problem with a scalar positive Radon measure as control, and in which the trajectories are of bounded variation. In order to deal with the impulsive differential equations, Rishel used the technique of `graph completion' that was formalized later by Bressan and Rampazzo in \cite{BreRam88}. In the latter article they dealt also with vector controls. The method of graph completion wa employed to obtain optimality conditions in Silva-Vinter \cite{SilVin97}, Pereira-Silva \cite{PerSil00}, Miller \cite{Millerbook}, Arutyunov et al. \cite{AruKarPer12}, among many others. 

Here we consider problems that admit trajectories of unbounded variation. More  precisely, the impulsive controls are no longer taken in the space of bounded variation, but in $L^1.$ The concept of solution we use is the one given in Bressan \cite{Bre87} for the scalar control case and \cite{BreRam91} for the vector case. They used a change of coordinates to transform the original system into a simple one where that could be regarded as a classical differential equation. A similar procedure has been used independently by Dykhta in \cite{Dyk94}.
We extend the Maximum Principle in \cite{BreRam91} to a formulation that includes a classical bounded control, and we obtain some higher order necessary conditions in terms of Lie bracket of the data functions.

The article is organized as follows. In Section \ref{SecSOL} we present the main assumptions, the definition of solution of an impulsive system and some properties. In Section \ref{SecTP} we define the transformed optimal control problem. In Section \ref{SecAE} we analyze the impulsive adjoint equation. We present a maximum principle in \ref{SecTPMP} and we derived necessary condition in \ref{SecNC}.


\section{A notion of solution }\label{SecSOL}

Consider the system \eqref{E}-\eqref{E0} written in its augmented form
\begin{align}
\label{AS}
&\begin{pmatrix}
\dot{x}\\
\dot{z}
\end{pmatrix}
= f(x,z)+ \sum_{i=1}^m \dot u^i\, {g}_i(x,z),\\
\label{AS0}&(x,z)(0)=(x_0,u_0),
\end{align}
where we added the dependent variable $z$ for the sake of simplicity in the presentation that follows.

\begin{assumption}[on the vector fields and control sets] 
\label{Ass}
\begin{itemize}
 \item[(i)] $U\subseteq \cR^m$ is the closure of a connected open set.
 \item[(ii)] $A\subseteq \cR^l$ is compact.
 \item[(iii)] For every $a\in A$, ${f}(\cdot,\cdot,a):\cR^n\times U\to\cR^n\times U$ is locally Lipschitz; and
 for every $(x,u)\in\cR^m,$ one has that ${f}(x,u,\cdot):A\to\cR^n\times U$ is continuous.
 \item[(iv)] There exists $M>0$ such that $|f(x,u,a)| \leq  M(1+|(x,u)|),$ for every $(x,u) \in \cR^n\times U.$       
\end{itemize}
\end{assumption}

This assumption guarantees the existence and uniqueness of the solution of the Cauchy problem \eqref{AS}-\eqref{AS0} for any initial condition $x_0\in\cR^n,$ 
$u\in \C^1(0,T;U)$ and $a\in L^1(0,T;A).$
Moreover, for every $M>0,$ there exists $N>0$ such that if $|x_0|+\| u\|_\infty<M$ then $\|x({x_0},u,a)\|_{\infty}<N.$  Here $x(x_0,u,a)$ denotes the unique solution of \eqref{AS}-\eqref{AS0} associated to $(x_0,u,a).$

Now we aim to give a definition of solution of the Cauchy problem \eqref{AS}-\eqref{AS0} for controls $u\in L^1.$  But in $L^1$ it can occur that two functions $u$ and $v$ are the same but $u(0)\neq v(0),$ and hence, special attention has to be payed. With this end, we introduce the concept of {pointwise defined} in the definition below. We consider not only the time $t=0,$ but any subset of $ [0,T].$

\begin{definition}
\begin{itemize}
\item[(i)] Let $I\subseteq [0,T]$. We say that two measurable maps $z,y:[0,T]\to\cR^d$ are {\em $I$-equivalent} if they coincide on every point of $I$ and almost everywhere on $[0,T] \backslash I.$
The class of equivalence of such maps is referred as {\em  pointwise defined on $I.$} 
We can identify a class of equivalence with one of its representatives as in the standard case.
\item[(ii)]
 If $F\subseteq\cR^d$, we use $L_I^1(0,T; F)$  to denote the subset of Lebesgue-integrable maps which are {\em pointwise defined} on $I$ and take values in $F$.
We say that a sequence $(y_k)\subset L_I^1$ {\em converges in $L_I^1$} to $y,$  if $y_k\to y$ in $L^1$ and $y_k(t)\to y(t)$ for every $t\in I.$
\end{itemize}
\end{definition} 

\begin{definition}[Generalized solutions pointwise defined on a subset $I$]
\label{DefGSpd}
Let $I\subseteq[0,T],$ with $0\in I.$ Consider $x_0\in \cR^n,$ and controls $u \in L^1_I(0,T;U)$ and $a\in L^1(0,T;A).$ A function $t\mapsto (x(t),z(t))$ of class $L^1_I$  is a {\em solution of \eqref{AS} pointwise defined on $I$} corresponding to the input $(x_0,u,a)$ if there exists a sequence of controls $(u_k)\subset \C^1([0,T],U)$ such that the functions $(x(a,u_k,\cdot),z_k(\cdot))$ have uniformly bounded values and converge  to $(x,u)$ in  $L^1_I.$ 
\end{definition}

In what follows we prove that the Definition \ref{DefGSpd} is a good definition. In other words, we show that under the Assumptions \ref{AssComm} and \ref{Ass} there is a unique generalized solution in $L^1$ of the Cauchy problem \eqref{AS}-\eqref{AS0} for each initial condition $(x_0,u_0)$ and control $(u,a)\in L^1\times L^1.$ Moreover, we prove that this concept of solution is robust, i.e. it is continuous with respect to the initial conditions and the controls $u$ and $a.$

The technique is introducing a diffeomorphism that transforms the equation \eqref{TS} into one where the impulsive part $\dot{u}$ has constant coefficients. Afterwards, the results of existence, uniqueness and continuity are proved for this simpler transformed case. Finally, it is shown that the same result holds for the general equation \eqref{E} by transformation.

\subsection{A Change of Coordinates}

Let us introduce a change of coordinate $\phi$ in the product space $\cR^n\times U$ that sends each vector field $g_\alpha$ into the  vector field $\ds\frac{\pr}{\pr z_\alpha}.$ This has the advantage that in the resulting system the derivative $\dot{u}$ multiplies constant vector fields.

For every $j=1,\dots,n,$ let $\varphi^j:\cR^{n}\times U \rightarrow \cR$ be given by
\be
\varphi^j(x,z) \doteq
\mathrm{Pr}^j \Big( \exp \left({- z^{k} g_{k}}\Big)
(x,z)\right),
\ee
where $\mathrm{Pr}^j : \cR^{n+m}\rightarrow \cR$ denotes the canonical projection on the $j-$th. coordinate. Set $\varphi \doteq(\varphi^1,\dots \varphi^n)$ and  consider the map $\phi:\cR^{n}\times U\rightarrow \cR^{n}\times U$ defined by
\be
\phi(x,z) \doteq (\varphi(x,z),z).
\ee

\begin{lemma}
The mapping $\phi$ is a diffeomorphism of $\cR^{n}\times U$ into itself. 
\end{lemma}

\begin{proof}

\end{proof}

\vs

The vector fields change with the differential of $\phi,$ i.e. for each $\alpha=1,\hdots,m,$ the transformed $f$ and $g_\alpha$ at $(\xi,\eta)=\phi(x,z)$ are given by
\be
F(\xi,\eta,v):= \nabla \phi (x,z) \, f(x,z,v), \quad
G_\alpha(\xi,\eta):= \nabla \phi (x,z) \, g_{\alpha}(x,z).
\ee

\begin{lemma}
\label{flowbox}
For every $\alpha=1,\dots, m$ one has 
\be
{G}_{\alpha}
= \frac{\pr }{\pr z_\alpha}.
\ee
\end{lemma}

\begin{proof} See Lemma 2.1 in \cite{BreRam91}.

\end{proof}

\begin{remark}
Lemma \ref{flowbox} just stated is actually the Simultaneous Straightening Out Theorem for commutative vector fields. The latter states that one can find a change of coordinates that transform a finite family of commutative smooth vector fields into constant vector fields. It is also known as Flow-box Theorem (see e.g. Abraham et al. \cite{AMR88} or Lang \cite{Lang}).
\end{remark}

On the other hand, notice that the last $m$ components of $F$ are zero. More precisely, $F$ can be written as $F=\begin{pmatrix} \tilde{F} \\ 0 \end{pmatrix}$ with
\be
\tilde F =\left( \frac{\pr \varphi^i}{\pr x^j} \tilde{f}^j \right) \frac{\pr}{\pr x^i}.
\ee
Consider hence the differential equation 
\be
\label{TS}
\begin{split}
 \dot\xi(t) &= \tilde{ F}(\xi(t),\eta(t),a(t)), \\
\dot\eta(t) &= \dot u(t),
\end{split}
\ee
with the initial conditions
\be
\label{TS0}
\begin{split}
 \xi(0) &= \varphi(x_0,u_0), \\
 \eta(0) &=u_0.
\end{split}
\ee

The following lemma states the equivalence of the system \eqref{AS}-\eqref{AS0} and the system obtained after the change of coordinates $\phi$ for measurable $a$ and smooth $u.$

\begin{lemma}[Equivalence of the equations for smooth $u.$]
\label{equivC1}
Let $(x,z,u,a)$ be a solution of the Cauchy problem \eqref{AS}-\eqref{AS0} with $u\in \C^1(0,T;U)$ and $a \in L^1(0,T;A).$
Then $(\xi,\eta,u,a)$ with
\bel{xieta}
(\xi,\eta)(t):= \phi(x(t),z(t))
\ee
is solution of \eqref{TS}-\eqref{TS0}.
Conversely, if $(\xi,\eta,u,a)$ is solution of \eqref{TS}-\eqref{TS0} with $u\in \C^1(0,T;U)$ and $a \in L^1(0,T;A),$  then $(x,z,u,a)$ is given by
\bel{transf}
(x(t),z(t)):= \phi^{-1}(\xi(t),\eta(t))
\ee
is solution of \eqref{AS}-\eqref{AS0}.
\end{lemma}

\begin{proof} 
The result follows immediately from the definition of $F$ and $G_i.$
\end{proof}

Observe now that in \eqref{TS} the impulsive part appears with a constant coefficient equal to 1. Hence, for every $u\in L^1,$ \eqref{TS} can be regarded as a classical differential equation by simple integration. More precisely, we consider \benl
 \eta=u, \quad
 \xi(t) = \xi(0) + \int_0^t \tilde{F} (\xi(s),\eta(s),a(s)) {\rm d}s.
\eenl

\begin{theorem}[Robustness for smooth $u$]
\label{ContTS}
\begin{itemize}
\item[(i)] The function 
\benl
a(\cdot)\mapsto \xi({\xi_0},u,a)(\cdot)
\eenl
 is continuous from $L^1(0,T;A)$ to $L^1(0,T;\cR^n),$ for each $x_0\in \cR^n,$ $u\in \C^1 (0,T;U).$
\item[(ii)] For $r>0$ and $\W\subset \C^1 (0,T;U),$ let  $K'\subset \cR^n$ such that the trajectories $\xi(\xi_0,u,a)$ have values inside $K'$ for all $\xi_0\in B_r(0),u\in \W,a\in L^1(0,T;A).$ Then there exists $M>0$ such that for every $\xi_0,\hat\xi_0 \in B_r(0)$ and $u,\uh \in \W,$
\be
\begin{split}
&|\xi(\xi_0,u,a)(\tau)-\xi(\hat\xi_0,\uh,a)(\tau)| + \int_0^T |\xi(\xi_0,u,a)(t)-\xi(\hat\xi_0,\uh,a)(t)| \dtt \\
&\leq
M\big[ |\xi_0-\hat\xi_0| + |u(0)-\uh(0)| + |u(\tau)-\uh(\tau)|
+\intT |u(t)-\uh(t)|\dtt
\big],
\end{split}
\ee
uniformly in $a \in L^1(0,T;A).$
\end{itemize}

\end{theorem}

In the proof of Theorem \ref{ContTS} we use the following result, which is a consequence of the Banach-Caccoppoli's 
\begin{lemma}
\label{BC}
Let $X$ be a Banach space, $M$ a metric space, $\Phi:M\times X \to X$ be a continuous function such that
\be
\label{BCeq1}
\| \Phi(m,x) - \Phi(m,y) \| \leq L \|x-y\|,\quad \forall m\in M,\ x,y\in X,
\ee
with $L<1.$ Then the following assertions hold.
\begin{itemize}
\item[(a)] For every $m\in M,$ there exists a unique $x(m)$ such that
\be
x(m) = \Phi(m,x(m)).
\ee
\item[(b)] The map $m\mapsto x(m)$ is continuous, and one has
\be
\| x(m)-x(m') \| \leq \frac{1}{1-L} \| \Phi(m,x(m)) - \Phi(m',x(m')) \|.
\ee
\end{itemize}
\end{lemma}

\begin{proof}

[of Theorem \ref{ContTS}] 
Assume for the moment that $F$ is globally Lipschitz in $(\xi,\eta)$ with constant $L.$ 
Fix $\tau \in [0,T],$ and consider the mapping 
\be
\begin{split}
\chi&(\xi_0,\eta_0,u,a,\xi,\eta):= \\
&\begin{pmatrix} \xi_0 \\ \eta_0 \end{pmatrix}
+ \int_0^t F(\xi(s),\eta(s),a(s)) {\rm d} s 
+\sum_{\alpha=1}^m [u^\alpha(t) - u^\alpha(0)]{\bf e}_{n+\alpha},
\end{split}
\ee
for $(\xi_0,\eta_0,u,a)\in M:=\cR^{n+m}\times \C^1 (0,T;U)\times L^1(0,T;A)$ and $(\xi,\eta) \in X:= \C^1(0,T;\cR^n)\times  \C^1(0,T;U)$
with the norm
\be
\| \omega \|_X:=
\frac{e^{-4TL}}{4L} |\omega(\tau)|
+\intT e^{-4tL}|\omega(t)|\dtt.
\ee
Observe that if $(\xi,\eta)=\chi(\xi_0,\eta_0,u,a,\xi,\eta),$ then $(\xi,\eta)$ is solution of \eqref{TS}-\eqref{TS0} with initial condition $(\xi_0,\eta_0).$ Then we are interested in applying Lemma \ref{BC} to $\chi.$

Let us prove that $\chi$ is continuous. Take two points $(\xi_0,\eta_0,u,a,\xi,\eta)$  and $(\xih_0,\etah_0,\uh,\ah,\xih,\etah)$ in the domain.
 One has the following estimations
\be
\label{estchi1}
\begin{split}
\|\chi (\xi_0,&\eta_0,u,a,\xi,\eta)-\chi(\xi_0,\eta_0,u,a,\xih,\etah)\|_X \\
= & \frac{e^{-4TL}}{4L} \left| \int_0^\tau (F(\xi(s),\eta(s),a(s)) - F(\xih(s),\etah(s),a(s)) ) {\rm d}s \right| \\
& +\intT e^{-4tL} \left| \int_0^t (F(\xi(s),\eta(s),a(s)) - F(\xih(s),\etah(s),a(s)) ) {\rm d}s \right| {\rm d} t \\ 
\leq & \left( \frac{e^{-4TL}}{4}+ \frac{e^{-4TL}}{-4} -\frac{1}{-4} \right) \intT |(\xi(s),\eta(s))-(\xih(s),\etah(s))| {\rm d}s \\
\leq& \frac{1}{4} \| (\xi,\eta) - (\xih,\etah) \|_X,
\end{split}
\ee
and
\be
\label{estchi2}
\begin{split}
\|\chi& (\xi_0,\eta_0,u,a,\xih,\etah)-\chi(\xih_0,\etah_0,\uh,\ah,\xih,\etah)\|_X \\
= & \frac{e^{-4TL}}{4L} \left| \begin{pmatrix} \xi_0-\xih_0 \\ \eta_0-\etah_0 \end{pmatrix} +\int_0^\tau (F(\xih(s),\etah(s),a(s)) - F(\xih(s),\etah(s),\ah(s)) ) {\rm d}s \right|  \\
&+ \intT e^{-4tL} \left| \int_0^t (F(\xih(s),\etah(s),a(s)) - F(\xih(s),\etah(s),\ah(s)) ) {\rm d}s \right| {\rm d} t \\ 
&+\frac{e^{-4TL}}{4L}\left| \sum_{\alpha=1}^m \big( u^\alpha(\tau)-u^\alpha(0)-\uh^\alpha(\tau)+\uh^\alpha(0)\big) {\bf e}_{n+\alpha}   \right|\\
&+\intT e^{-4tL} \left| \int_0^t  \sum_{\alpha=1}^m \big(u^\alpha(s)-u^\alpha(0)-\uh^\alpha(s)+\uh^\alpha(0)\big){\bf e}_{n+\alpha} {\rm d} s \right| {\rm d} t.
\end{split}
\ee
Thus, for each $(\xih_0,\etah_0,\uh,\ah,\xih,\etah)$ and for every $\eps>0,$ there exists $\delta>0$ such that if 
\benl
\begin{split}
&|(\xi_0,\eta_0)-(\xih_0,\etah_0)| +|u(0)-\uh(0)| \\
&\quad +|u(\tau)-\uh(\tau)|+\|u-\uh\|_1+ \|a-\ah\|_1 + \| (\xi,\eta) - (\xih,\etah) \|_X < \delta
\end{split}
\eenl 
then
\benl
\|\chi (\xi_0,\eta_0,u,a,\xi,\eta)-\chi(\xih_0,\etah_0,\uh,\ah,\xih,\etah)\|_X <\eps,
\eenl
and hence $\chi$ in continuous. Observe that the modulus of continuity does not depend on $\tau,$ but on $|u(\tau)-\uh(\tau)|$ and, therefore, the same estimation holds for every $\tau \in [0,T].$
Moreover, in view of \eqref{estchi1}, the inequality \eqref{BCeq1} holds as well. Therefore we can apply Lemma \ref{BC} to $\chi$ which yields the desired result for $F$ globally Lipschitz.

In case $F$ is only locally Lipschitz, define
\benl
\hat{F} = 
\left\{
\ba{cl}
F & {\rm on}\ K'\times U, \\
0 & {\rm on}\ (K'\times U)^c,
\ea
\right.
\eenl
and follow previous procedure. The desired result follows.

\end{proof}

\subsection{Properties of the impulsive system}

The analogous of Theorem \ref{ContTS} can be proved for the impulsive system \eqref{AS} by means of the transformation $\phi.$ Hence we get the following result.

\begin{theorem}
\label{ContC1}
\begin{itemize}
\item[(i)] The function $a(\cdot)\mapsto x(x_0,u,a)(\cdot)$ is continuous from $L^1(0,T;A)$ to $L^1(0,T;\cR^n),$ for each $x_0\in \cR^n,$ $u\in \C^1(0,T;U).$
\item[(ii)] For $r>0$ and $\W\subset \C^1(0,T;U),$ let  $K'\subset \cR^n$ such that the trajectories $x(x_0,u,a)$ have values inside $K'$ for all $x_0\in B_r(0),u\in \W,a\in L^1(0,T;A).$ Then there exists $M>0$ such that for every $x_0,\xh_0 \in B_r(0)$ and $u,\uh \in \W,$
\bel{EstC1}
\begin{split}
&|x(x_0,u,a)(\tau)-x(\xh_0,\uh,a)(\tau)| + \int_0^T |x(x_0,u,a)(t)-x(\hat x_0,\uh,a)(t)| \dtt \\
&\leq
M\big[ |x_0-\hat x_0| + |u(0)-\uh(0)| + |u(\tau)-\uh(\tau)|
+\intT |u(t)-\uh(t)|\dtt
\big],
\end{split}
\ee
uniformly in $a \in L^1(0,T;A).$
\end{itemize}
\end{theorem}

From previous Theorem \ref{ContC1} and Definition \ref{DefGSpd} we get the following result.
\begin{corollary}
For each $(x_0,u,a)\in \cR^n \times L^1_{\{0\}}(0,T;U) \times  L^1(0,T;A),$ there exists a unique generalized solution $\xi \in L^1_{\{0\}} (0,T;\cR^n), \eta=u,$ of the impulsive Cauchy problem \eqref{AS}-\eqref{AS0}. 
\end{corollary}

\subsection{Generalized solution pointwise defined everywhere}
\label{RemConst}
In the case where $u$ is defined pointwise on $[0,T],$ the trajectory $x(x_0,u,a)(\cdot)$ can also be determined pointwise following the procedure described next. Let $\tau \in [0,T],$ and consider a sequence $(w_k^\tau) \subset \C^1$ such that $w_k^\tau(0) = u(0),$ $w_k^\tau(\tau) = u(\tau)$ and $w_k^\tau \to u$ in $L^1(0,T;U).$  The estimate \eqref{EstC1} implies that $x(x_0,w_k^\tau,a)(\cdot)$ tends to $x(x_0,u,a)(\cdot)$ in $L^1$ and that $x(x_0,w_k^\tau,a)(\tau)$ has limit. Denote this limit by $x(\tau).$ Note that two different sequences $w_k^\tau$ and $\wh_k^\tau$ yield the same $x(\tau)$ by  \eqref{EstC1}. Thus $x$ is well-defined. In the sequel we prove that $x$  is a generalized solution of \eqref{AS}-\eqref{AS0}.
In fact, for any $t\in [0,T],$ one can extract a subsequence $(w_{k'}^t)$ from $(w_k^t)$ which converges pointwise to $u$ on the complement of a set $\N$ of measure zero. For any $\tau \in [0,T] \backslash \N,$ by \eqref{EstC1} one has
\benl
\begin{split}
|x(\tau)-&x(x_0,w_{k'}^t,a)(\tau)| \\
&\leq |x(\tau)-x(x_0,w_{k'}^\tau,a)(\tau)|  + |x(x_0,w_{k'}^\tau,a)(\tau)-x(x_0,w_{k'}^t,a)(\tau)|\\
&\quad + \int_0^T |x(x_0,w_{k'}^\tau,a)(s)-x(x_0,w_{k'}^t,a)(s)| {\rm d}s \\
&\leq|x(\tau)-x(x_0,w_{k'}^\tau,a)(\tau)| + 
M\big[ |w_{k'}^\tau(\tau)-w_{k'}^t(\tau)|\\
&\quad+\intT |w_{k'}^\tau(s) - w_{k'}^t(s)    |{\rm d} s
\big].
\end{split}
\eenl
The right hand-side goes to 0 since $w_{k'}^\tau (\tau) \to U(\tau).$ Thus, $x(x_0,w_{k'}^t,a) \to x$ almost everywhere. Since $x(x_0,w_{k'}^t,a)$ have uniformly bounded values, it follows that they converge to $x$ in $L^1.$ Therefore, $(x,z=u)$ is a generalized solution of \eqref{AS}-\eqref{AS0}.

Given $x_0\in \cR^n,$ $u\in L_{[0,T]}^1$ and $a\in L^1(0,T;A),$ there is only one solution $(\xi,\eta=u)$ of \eqref{TS}-\eqref{TS0}, for which $\xi$ is an absolutely continuous function. Note that
\benl
\begin{split}
x(\tau) &= \lim_{k\to\infty} x(x_0,w_k^\tau,a)(\tau) \\
& = \lim_{k\to\infty} \varphi^{-1}(\xi(\varphi(x_0,u_0),w_k^\tau,a)(\tau),w_k^\tau (\tau))
= \varphi^{-1} (\xi(\tau),u(\tau)).
\end{split}
\eenl
Hence, we get.

\begin{proposition}
\label{Propequiv}
For each $(x_0,u,a) \in \cR^n \times \LT(0,T;U) \times L^1(0,T;A),$ there is a unique solution $(x,z)$ of \eqref{AS}-\eqref{AS0} that is pointwise defined in $[0,T].$ Moreover, it is given by the formula
\benl
(x(\tau),z(\tau)) = \phi^{-1}(\xi(\tau),u(\tau)),\quad \text{for}\ \text{all}\ \tau \in [0,T],
\eenl
where $\xi \in AC(0,T;\cR^n)$ is the unique solution of \eqref{TS}-\eqref{TS0} corresponding to $(x_0,u,a)$ and $\eta=u.$
\end{proposition}

\begin{theorem}[Robustness of the impulsive system]
The assertions  in Theorem \ref{ContC1} hold  when we  consider the controls u in $\LT(0,T;U).$ 
\end{theorem}

\begin{proof}
It is a immediate consequence of Theorem \ref{ContC1} and the Definition \ref{DefGSpd} of generalized solution.
\end{proof}

\subsection{Statement of the Optimal Control Problem}

Now we are ready to state in a proper way the optimal control problem we deal with. 
Let $\gamma:\cR^{n+m} \to \cR$ be a smooth function. Denote by $\P$ the (impulsive) optimal control problem of finding $(x_0,u,a)\in \cR^n \times \LT(0,T;U) \times L^1(0,T;A)$ that minimizes 
\benl
\gamma(x(T),z(T)),
\eenl
where $(x,z)$ is the generalized solution of \eqref{AS}-\eqref{AS0} associated to  $(x_0,u,a).$

\section{The Transformed Optimal Control Problem}\label{SecTP}

Next we introduce an auxiliary optimal control problem in the transformed variables that will be used afterwards to derive optimality conditions for $\P.$ Denote by $\P'$ the problem consisting of minimizing
\be
\Psi (\xi(T),\eta(T)),
\ee
over the trajectories of the system \eqref{TS}-\eqref{TS0} with controls $u\in \LT(0,T;U).$
Here the function $\Psi: \cR^{n+m} \rightarrow \cR$ is defined by
\be
\Psi(\xi,\eta):= \gamma (\phi^{-1}(\xi,\eta)).
\ee

The following result is a straightforward consequence of Proposition \ref{Propequiv}.

\begin{proposition}
A triple $(x_0,u^*,a^*) \in \cR^n \times \LT(0,T;U) \times L^1(0,T;A)$ is optimal for $\P$ if and only if it is optimal for $\P'.$ Moreover, $\P$ and $\P'$ have the same optimal values.
\end{proposition}


\section{The Adjoint Equation}\label{SecAE}

In this section we show that the adjoint equation associated to \eqref{AS} is commutative. Let us first establish a technical lemma that will yield the desired commutativity afterwards.

Consider $a$ and $b$ two vector fields of class $\C^2$ from $\cR^{N}$ to $\cR^N,$ and define
\bel{A}
\begin{split}
\A:\cR^{2N}&\rightarrow \cR^{2N}\\
(y,w)&\rightarrow \A(y,w):=
\begin{pmatrix}
a(y)\\
-\nabla a(y)\tras \cdotp w
\end{pmatrix},
\end{split}
\ee
and
\bel{B}
\begin{split}
\B:\cR^{2N}&\rightarrow \cR^{2N}\\
(y,w)&\rightarrow \B(y,w):=
\begin{pmatrix}
b(y)\\
-\nabla b(y)\tras \cdotp w
\end{pmatrix}.
\end{split}
\ee

\begin{lemma}\label{ABcomm}
Let $a,b:\cR^N\rightarrow \cR^N$ be two vector fields of class $\C^2,$ such that $[a,b]=0.$ Then, the vector fields defined in \eqref{A}-\eqref{B} commute as well, or equivalently, $[\A,\B]=0.$
\end{lemma}

\begin{proof}
Here the Einstein notation is used, which implies summation over repeated indexes.
For $k=1,\dots,N,$ we have
\be
[\A,\B]^k=
\frac{\partial b^k}{\partial y^s}\, a^s -
\frac{\partial a^k}{\partial y^s}\, b^s=[a,b]^k=0,
\ee
and
\be
\begin{split}
[\A,\B]^{N+k}&=
-\frac{\partial^2 b^\ell}{\partial y^s \partial y^k}\, w^{\ell}\, a^s +\frac{\partial b^s}{\partial y^k}\frac{\partial a^{\ell}}{\partial y^s} w^{\ell}
+\frac{\partial^2 a^\ell}{\partial y^s \partial y^k}\, w^{\ell}\,b^s -\frac{\partial a^s}{\partial y^k}\frac{\partial b^{\ell}}{\partial y^s} w^{\ell}\\
&=
\frac{\partial}{\partial y^k}\left\{-\frac{\partial b^{\ell}}{\partial y^s}\,a^s + \frac{\partial a^{\ell}}{\partial y^s}\, b^s   \right\}w^{\ell}=\frac{\partial}{\partial y^k}[b,a]^{\ell} w^{\ell}=0.
\end{split}
\ee
Thus, the result follows.
\end{proof}

Consider now the augmented system \eqref{AS}-\eqref{AS0} together with its associated {\it adjoint equation}
\begin{align}
\label{AS2}\begin{pmatrix}
\dot{x}\\ \dot{z}
\end{pmatrix} &= {f}(x,z,a) + \ds\sum_{i=1}^m \dot u^i g_i(x,z),\\
\label{AE}(\dot{p}_1, \dot{p}_2)
&=
-
({p_1}, {p_2})
\,\cdotp \left( \nabla_{(x,u)}{f}(x,u,a) + \sum_{i=1}^m \dot{u}^{i} \nabla {g}_i (x,u)\right),
\end{align}
and the endpoint conditions
\bel{AS02}
(x(0),z(0))=(x_0,u_0),
\ee
\bel{AET}
({p_1}(T),{p_2}(T) )= \nabla \gamma (x(T),u(T)).
\ee
The vector fields that are coefficients of $\dot{u}$ in \eqref{AS2}-\eqref{AE} are given by
\be
\G_i (x,z,p_1,p_2):=
\begin{pmatrix}
g_i(x,z) \\
\nabla g_i\tras(x,z) \cdotp
\begin{pmatrix}
{p_1}\tras\\ {p_2}\tras
\end{pmatrix}
\end{pmatrix}.
\ee
Applying Lemma \ref{ABcomm} we get that,  for each pair $i,j,$  $[\G_i,\G_j]=0$ since $[g_i,g_j]=0.$ Hence, a concept of solution equivalent to the one given for the augmented impulsive system \eqref{AS}-\eqref{AS0} can be given to \eqref{AE},\eqref{AET}. We will refer to it as the {\it generalized solution of the adjoint equation.}
Furthermore, we can also relate \eqref{AS2}-\eqref{AET} with the adjoint equation associated to the transformed equation \eqref{TS}-\eqref{TS0} via a change of variables.
With this aim, consider the adjoint system associated to \eqref{TS}-\eqref{TS0},
\bel{ATE}
\begin{split}
 \dot\xi(t) &= \tilde{ F}(\xi(t),\eta(t),a(t)), \\
\dot\eta(t) &= \dot u(t),\\
\dot\pi_1(t)&=-\pi_1(t)\cdotp\nabla_{\xi}\tilde{ F}(\xi(t),\eta(t),a(t)),\\
\dot\pi_2(t)&=-\pi_1(t)\cdotp\nabla_{\eta}\tilde{ F}(\xi(t),\eta(t),a(t)),
\end{split}
\ee
with endpoint conditions given by
\bel{ATE0T}
\begin{split}
 \xi(0) &= \varphi(x_0,u_0), \\
 \eta(0) &=u_0,\\
(\pi_1(T),\pi_2(T)) &= \nabla \Psi(\xi(T),\eta(T)).
\end{split}
\ee

\begin{proposition}[Generalized solution of the adjoint equation]
Given $(x,z,u)$ a generalized solution of \eqref{AS}-\eqref{AS0} defined pointwise everywhere, let $(\xi,\eta,a,u)$ be its transformation through $\phi$ and $(\pi_1,\pi_2)$
be the solution of \eqref{ATE}-\eqref{ATE0T} in the classical sense.
Then, the generalized solution $p$ of \eqref{AE},\eqref{AET} verifies
\be
p=\pi \cdotp \nabla \phi (x,u).
\ee
\end{proposition}
The latter result is an easy consequence of the change of coordinates. We pass now to the second part of the article where we provide a set of necessary conditions for optimality.


\section{The Maximum Principle for the Transformed problem}\label{SecTPMP}

We recall here two theorems due to Bressan-Rampazzo \cite{BreRam91}. The first statement is Theorem \ref{PMP} that provides a maximum principle for $\P,$ and it is a consequence of the second result presented in Theorem \ref{TPMP} that is a maximum principle for the transformed problem $\P'.$

Theorem \ref{PMP} is a consequence of the following result.

\begin{theorem}
[Maximum Principle for the Transformed Problem]
\label{TPMP}
Let $(\xi^*,u^*,a^*)$ be an optimal for problem $\P',$ and let $((\xi^*,u^*),(\pi^*_1,\pi^*_2))$ denote the  solution of the adjoint equation \eqref{ATE}-\eqref{ATE0T}. Then,
\be
\label{TPMPcond1}
\pi_1(t)\left( \tilde{F}(\xi^*(t),u,a)-\tilde{F}(\xi^*(t),u^*(t),a^*(t))\right) \geq 0,
\ee
for almost every $t\in [0,T]$ and for every $u \in U,a\in A.$ Moreover, if for any $t\in [0,T],$ $\nu:[t,T]\rightarrow U$ is a map as in Theorem \ref{PMP}, then
\be
\label{TPMPcond2}
\pi_2(t) \cdotp \nu(t) \geq - \int_{[t,T]} \pi_2 \cdotp \dot{\nu}.
\ee
\end{theorem}

\begin{remark}
Actually, the conditions \eqref{PMPcond1} and \eqref{TPMPcond1} hold on the set of Lebesgue points of $(u^*,a^*)$ and hence, since $(u^*,a^*)$ is a $L^1-$function, it holds almost everywhere on $[0,T].$
\end{remark}

\begin{proof}

[of Theorem \ref{TPMP}]

Observe that  the classical Pontryagin Maximum Principle (see \cite{Pontryagin}) can be applied to problem $\P',$ and it yields the minimum condition \eqref{TPMPcond1}. The latter holds at every Lebesgue point of $(u^*,a^*)$  and hence, since $(u^*,a^*)$  is in $L^1,$ \eqref{TPMPcond1} holds almost everywhere on $[0,T].$

In order to prove \eqref{TPMPcond2}, let $\sigma \in [0,\sigma_0]$ for a small positive $\sigma_0$ and consider the controls $u_{\sigma}$ given by
\be
 u_{\sigma}(\tau)=
\left\{
\ba{ll}
u^*(\tau),\quad & \text{if}\ \tau \in [0,t), \\
u^*(\tau)+\sigma \nu(\tau),\quad & \text{if}\ \tau \in [t,T].
\ea
\right.
\ee
Here $\nu:[t,T] \rightarrow \cR^m$ is a function satisfying the hypotheses of Theorem \ref{PMP}.
Denote $\xi_{\sigma}$ the solution of the transformed system \eqref{TS} corresponding to $u_{\sigma}.$
 It turns out that (see e.g. \cite{Hartman} or \cite[Theorem 10.2, Chapter II]{FleRis}),
\be
\left.\frac{\d}{\d\sigma}\right|_{\sigma=0^+} \xi_{\sigma}(T) = \omega(T),
\ee
where $\left.\frac{d}{d\sigma}\right|_{\sigma=0^+}$ refers to the right derivative at $\sigma=0$ and $\omega : [t,T] \rightarrow \cR^n$ is the solution of
\be
\label{eqomega}
\left\{
\begin{split}
\dot{\omega}(\tau) &= \nabla_{\xi} \tilde{F} (\xi^*(\tau),u^*(\tau),a^*(\tau))\, \omega(\tau) + \nabla_{\eta} \tilde{F} (\xi^*(\tau),u^*(\tau)a^*(\tau))\, \nu(\tau), \\
\omega(t) & =0.
\end{split}
\right.
\ee
From \eqref{ATE} and \eqref{eqomega} we get
\be
\frac{\d}{\d \tau}
[(\pi_1,\pi_2) \cdotp (\omega,\nu) ]= \pi_2 \cdotp \dot{\nu},
\ee
and thus, the relation
\be
\label{difpi}
(\pi_1(T),\pi_2(T)) \cdotp (\omega(T),\nu(T))
-(\pi_1(t),\pi_2(t)) \cdotp (\omega(t),\nu(t))
= \int_{[t,T]} \pi_2 \cdotp \dot{\nu}
\ee
follows. Since $u^*$ is optimal,
\be
\left.\frac{\d}{\d\sigma}\right|_{\sigma=0^+} \Psi(\xi_{\sigma}(T),u_{\sigma}(T)) \geq 0,
\ee
and therefore,
\be
\begin{split}
0 & \leq \left.\frac{\d}{\d\sigma}\right|_{\sigma=0^+} \Psi(\xi_{\sigma}(T),u_{\sigma}(T)) \\
& = \nabla \Psi (\xi^*(T),u^*(T)) \cdotp \left.\frac{\d}{\d\sigma}\right|_{\sigma=0^+} (\xi_{\sigma}(T),u_{\sigma}(T)) \\
& = (\pi_1(T),\pi_2(T)) \cdotp (\omega(T),\nu(T)).
\end{split}
\ee
Considering \eqref{difpi} we get
\be
\int_{[t,T]} \pi_2 \dot{\nu} + (\pi_1,\pi_2) \cdotp (\omega,\nu) (t)
= (\pi_1(T),\pi_2(T)) \cdotp (\omega(T),\nu(T)) \geq 0.
\ee
Hence, since $\omega(t)=0,$ we obtain the inequality  \eqref{TPMPcond2}. This concludes the proof.

\end{proof}

\section{Necessary Optimality conditions \\ involving Lie brackets}\label{SecNC}

\begin{definition}
For every $(x,u_1,a)\in \cR^n\times U\times A$ and $u_2\in U,$ the $n-$dimensional vector
\benl
\T_{u_2} \tilde{f}(x,u_1,a) = \nabla_{x} \varphi(\varphi(x,u_1-u_2),u_2-u_1) \cdotp \tilde{f}(\varphi(x,u_1-u_2),u_2,a),
\eenl
is called the $u_2-$transport of $\tilde{f}$ at $(x,u_1,a).$
\end{definition}

\begin{theorem}[Maximum Principle]
\label{PMP}
Let $(u^*,a^*,x^*)$ be an optimal control for $\P,$ and let $((x^*,z^*),(p_1,p_2))$ denote the generalized solution of the adjoint system \eqref{AS2}-\eqref{AET}. Then, \begin{itemize}
\item[(i)] The inequality
\be
\label{PMPcond1}
p_1(t) \cdotp \left(\T_u \tilde{f}(x^*(t),u^*(t),a^*(t))-\tilde{f}(x^*(t),u^*(t),a^*(t) \right) \geq 0
\ee
holds for a.a. $t\in [0,T]$ and for every $u\in U.$ 
\item[(ii)]
For every $a\in A$ and a.a. $t\in [0,T],$
\be
p_1(\tilde{f}(x^*(t),u^*(t),a)-\tilde{f}(x^*(t),u^*(t),a^*(t)) \geq 0.
\ee
\item[(iii)]
Moreover, for any $t\in [0,T],$ let $\nu: [t,T]\rightarrow U$ be a bounded variation map, that is right continuous at $t$ and left continuous at $T,$ and such that $(u^*+\sigma \nu)(\tau)\in U,$ for a.a. $\tau \in [t,T]$ and for $\sigma \in [0,\sigma_0].$ Then, one has
\be
\label{PMPcond2}
p_1(t) \cdotp \sum_{i=1}^m\tilde{g}_i (x^*(t),u^*(t)) \nu^i(t) + p_2(t) \cdotp \nu(t) \geq \int_{[t,T]} p_2 \dot\nu.
\ee
Here the integral on the right hand-side is the integral of $p_2$ with respect to the vector Radon measure $\dot{\nu}.$
\end{itemize}
\end{theorem}

\begin{theorem}[Necessary conditions involving Lie brackets]
\label{NCLie}
Let $(x^*,u^*,a^*)$ be optimal for $\P,$  and let $i=1,\dots m$ be an index.  Then the following statements hold.
\begin{itemize}
\item[(i)] Let $t \in [0,T)$ be any time such that there exists $\sigma>0$ sufficiently small such that $u^*(t) +\sigma {\bf e_i} \in U$ a.e. on $[t,T].$ Then
\be
\label{NCi}
p(t) \cdotp g_i (x^*(t),u^*(t)) \leq 0.
\ee
\end{itemize}
Furthermore, for a.a. $t\in  [0,T],$ it holds:
\begin{itemize}
 \item[(ii)] If there exists $\sigma_0>0$ sufficiently small such that $u^*(t) +\sigma {\bf e_i} \in U$ for all $\sigma \in [0,\sigma_0],$ then
\be
\label{NCii}
p(t) \cdotp [g_i,f] (x^*(t),u^*(t),a^*(t)) \geq 0.
\ee
\item[(iii)] If there exist $h\in \cR^m$ and $\sigma_0>0$ sufficiently small such that $u^*(t) \pm \sigma {\bf e_i} \in U$ for all $\sigma \in [0,\sigma_0],$ then
\be
\label{NCiii}
p(t) \cdotp \sum_{j,k=1}^m h_jh_k [g_j,[g_k,f] \,] (x^*(t),u^*(t),a^*(t)) \geq 0.
\ee
In other words,
\be
h^\top Qh \geq 0,
\ee
where $Q$ is a symmetric matrix with entries
\benl
Q_{jk}(t) := p(t) \cdotp [g_j,[g_k,f] \,] (x^*(t),u^*(t),a^*(t)) .
\eenl
\end{itemize}
Here all the Lie brackets are computed in the variable $(x,u).$
\end{theorem}

\begin{remark}
\begin{itemize}
\item[a)] Notice that the assumption on $u^*$ in item (i) is stronger that the ones done in (ii) and (iii). This is due to the fact that different control variations are employed for obtaining different necessary conditions.
\item[b)] Observe that if $U=\cR^m,$ then the condition \eqref{NCi} implies that at every $t\in [0,T],$
\be
p(t) \cdotp [g_i,f] (x^*(t),u^*(t),a^*(t)) = 0.
\ee
The latter condition was also obtained by Silva and Vinter in \cite{SilVin97}, for the case when $u$ is a scalar bounded variation function.
\end{itemize}
\end{remark}


\begin{proof}
Let $i,j$ and $t$ be as in the statement of the theorem. We shall start by proving item (i). Let $\nu:[t,T] \rightarrow \cR^m$ be given by
\be
\nu_i \equiv 1,\quad \nu_k \equiv 0\ \text{for}\ \text{all}\ k\neq i.
\ee
Then $\nu$ verifies the hypotheses of Theorem \ref{PMP}. For this particular $\nu,$ the condition \eqref{PMPcond2} yields
\be
\label{p1tildeg}
p_1(t) \cdotp \tilde{g}_i(x^*(t),u^*(t)) + p_2(t) \cdotp {\bf e_i} \leq 0,
\ee
where ${\bf e_i}$ is the $i-$th. canonical vector in $\cR^m.$ Finally, notice that \eqref{p1tildeg} can be rewritten as \eqref{NCi} and hence  (i) follows.

In order to prove (ii) recall the condition \eqref{TPMPcond1} of Theorem \ref{TPMP}. Set $u=u^*(t)+\sigma {\bf e_i},$ and observe that \eqref{TPMPcond1} implies
\be
\label{iieq1}
\pi_1(t) \frac{\tilde{F}(\xi^*(t),u^*(t) +\sigma {\bf e_i},a^*(t)) - \tilde{F}(\xi^*(t),u^*(t),a^*(t))}{\sigma} \geq 0.
\ee
By taking the limit as $\sigma \rightarrow 0^+$ we get
\be
\label{iieq2}
\pi_1(t) \cdotp\frac{\pr \tilde{F}}{\pr u_i}(\xi^*(t),u^*(t),a^*(t)) \geq 0.
\ee
On the other hand, notice that
\be
\pi_1 \cdotp\frac{\pr \tilde{F}}{\pr u_i} = \pi \cdotp \frac{\pr {F}}{\pr u_i} = \pi \cdotp [g_i,F].
\ee
Hence, \eqref{iieq2} is identical to
\be
\pi(t)\cdotp [G_i,F] (\xi^*(t),u^*(t),a^*(t)) \geq 0,
\ee
which coincides with \eqref{NCii} in the original coordinates and thus (ii) is proved.

We shall now prove (iii). First observe that for $\sigma>0$ sufficiently small it holds
\be
\pi_1(t) \frac{\tilde{F}(\xi^*(t),u^*(t) +\sigma h,a^*(t)) - \tilde{F}(\xi^*(t),u^*(t),a^*(t))}{\sigma} \geq 0,
\ee
and the opposite inequality holds for $u^*(t) - \sigma h.$ Thus,
\be
\pi_1(t) \nabla_h \tilde{F}(\xi^*(t),u^*(t),a^*(t)) = 0,
\ee
where $\nabla_h$ denotes the directional derivative in the direction $h.$
Consider now the second order Taylor expansion
\benl
\begin{split}
\tilde{F}&(\xi^*(t),u^*(t) + \sigma {h},a^*(t)) = \tilde{F}(\xi^*(t),u^*(t),a^*(t)) \\& + \sigma \nabla_h\tilde{F}(\xi^*(t),u^*(t),a^*(t))  
+ \sigma^2 \nabla^2_{h,h} \tilde{F}(\xi^*(t),u^*(t),a^*(t)) + o(\sigma^2).
\end{split}
\eenl
By multiplying by $\pi_1(t)$ and dividing by $\sigma^2$ we get
\be
0\leq \pi_1(t) \cdotp \frac{\tilde{F}(\xi^*(t),u^*(t) + \sigma {h},a^*(t))}{\sigma^2} = \nabla^2_{h,h} \tilde{F}(\xi^*(t),u^*(t),a^*(t)) + o(1),
\ee
where the first inequality holds by \eqref{TPMPcond1}.
Taking the limit as $\sigma$ goes to 0 yields
\be
\label{iiieq1}
\pi_1(t) \cdotp \nabla^2_{h,h} \tilde{F}(\xi^*(t),u^*(t),a^*(t))  \geq 0.
\ee
Notice that
\be
\label{iiieq2}
\pi_1 \cdotp \nabla^2_{h,h} \tilde{F}
= \pi \cdotp \sum_{k=1}^m h_k\nabla_{u_k} \sum_{j=1}^m h_j \nabla_{u_i} {F} =  \pi \cdotp \sum_{j, k=1}^m h_j h_k [G_j,[G_k,F]].
\ee
Equation \eqref{iiieq2} written in the original coordinates together with the inequality \eqref{iiieq1} imply \eqref{NCiii}.
Finally, we shall prove the symmetry of the matrix $Q.$ Notice that, since $F$ is of class $\C^2,$ then $ \nabla_{u_k}\nabla_{u_j} F = \nabla_{u_j}\nabla_{u_k} F.$ By multiplying by $\pi$ and rewriting in the original coordinates, the symmetry follows. \footnote{The symmetry of $Q$ follows also from the Jacobi identity and the commutativity of $g_i.$}
This completes the proof.
\end{proof}

\bibliographystyle{plain}
\bibliography{impulsive}

\begin{thebibliography}{10}

\bibitem{AMR88}
R.~Abraham, J.~E. Marsden, and T.~Ratiu.
\newblock {\em Manifolds, tensor analysis, and applications}, volume~75 of {\em
  Applied Mathematical Sciences}.
\newblock Springer-Verlag, New York, second edition, 1988.

\bibitem{AruKarPer12}
A.~V. Arutyunov, D.~Yu. Karamzin, and F.~Pereira.
\newblock Pontryagin's maximum principle for constrained impulsive control
  problems.
\newblock {\em Nonlinear Anal.}, 75(3):1045--1057, 2012.

\bibitem{Bre87}
A.~Bressan.
\newblock On differential systems with impulsive controls.
\newblock {\em Rend. Sem. Mat. Univ. Padova}, 78:227--235, 1987.

\bibitem{BreRam91}
A.~Bressan, Jr. and F.~Rampazzo.
\newblock Impulsive control systems with commutative vector fields.
\newblock {\em J. Optim. Theory Appl.}, 71(1):67--83, 1991.

\bibitem{BreRam88}
Alberto Bressan and Franco Rampazzo.
\newblock On differential systems with vector-valued impulsive controls.
\newblock {\em Boll. Un. Mat. Ital. B (7)}, 2(3):641--656, 1988.

\bibitem{Dyk94}
V.~A. Dykhta.
\newblock The variational maximum principle and quadratic conditions for the
  optimality of impulse and singular processes.
\newblock {\em Sibirsk. Mat. Zh.}, 35(1):70--82, ii, 1994.

\bibitem{FleRis}
Wendell~H. Fleming and Raymond~W. Rishel.
\newblock {\em Deterministic and stochastic optimal control}.
\newblock Springer-Verlag, Berlin, 1975.
\newblock Applications of Mathematics, No. 1.

\bibitem{Hartman}
P.~Hartman.
\newblock {\em Ordinary differential equations}.
\newblock John Wiley \& Sons Inc., New York, 1964.

\bibitem{Lang}
Serge Lang.
\newblock {\em Differential and {R}iemannian manifolds}, volume 160 of {\em
  Graduate Texts in Mathematics}.
\newblock Springer-Verlag, New York, third edition, 1995.

\bibitem{Millerbook}
B.M. Miller and E.Y. Rubinovich.
\newblock {\em Impulsive control in continuous and discrete-continuous
  systems}.
\newblock Kluwer Academic/Plenum Publishers, New York, 2003.

\bibitem{PerSil00}
F.L. Pereira and G.N. Silva.
\newblock Necessary conditions of optimality for vector-valued impulsive
  control problems.
\newblock {\em Systems Control Lett.}, 40(3):205--215, 2000.

\bibitem{Pontryagin}
L.~Pontryagin, V.~Boltyanski, R.~Gamkrelidze, and E.~Michtchenko.
\newblock {\em The Mathematical Theory of Optimal Processes}.
\newblock Wiley Interscience, New York, 1962.

\bibitem{Ris65}
Raymond~W. Rishel.
\newblock An extended {P}ontryagin principle for control systems whose control
  laws contain measures.
\newblock {\em J. Soc. Indust. Appl. Math. Ser. A Control}, 3:191--205, 1965.

\bibitem{SilVin97}
G.~N. Silva and R.~B. Vinter.
\newblock Necessary conditions for optimal impulsive control problems.
\newblock {\em SIAM J. Control Optim.}, 35(6):1829--1846, 1997.

\end{thebibliography}

\end{document}